\documentclass{article}
 
\usepackage{eurosym}
\usepackage{amssymb}
\usepackage{amsmath}
\usepackage{amsfonts}
\usepackage{graphicx}
\usepackage[pagebackref=true]{hyperref}
\usepackage{color}

\newtheorem{theorem}{Theorem}[section]
\newtheorem{corollary}[theorem]{Corollary}
\newtheorem{example}[theorem]{Example}
 \newtheorem{lemma}[theorem]{Lemma}

  \newtheorem{definition}[theorem]{Definition}

\newenvironment{proof}[1][Proof]{\noindent\textbf{#1.} }{\ \rule{0.5em}{0.5em}}
\begin{document}

\author{Do Trong Hoang \\
Faculty of Mathematics and Informatics \\
Hanoi University of Science and Technology\\
1 Dai Co Viet, Bach Mai, Hanoi, Vietnam\\
hoang.dotrong@hust.edu.vn \and Vadim E. Levit \\
Department of Mathematics\\
Ariel University, Israel\\
levitv@ariel.ac.il \and Eugen Mandrescu \\
Department of Computer Science\\
Holon Institute of Technology, Israel\\
eugen\_m@hit.ac.il }
\title{Structural properties and characterizations of $\mathbf{W}_{p}$ class%
}
\date{}
\maketitle

\begin{abstract}
 We establish new characterizations of graphs belonging to the   $\mathbf{W}_p$ class. In addition, we characterize locally triangle-free  $\alpha$-critical graphs in this class. As a consequence, our results yield a partial answer to a question raised by Plummer \cite{Plum93} in the case  $p=2$. 
 \end{abstract}

\textbf{Keywords}: $\alpha$-critical graph;    $\mathbf{W}_{p}$ graph;  well-covered graph.   \newline
\noindent \textbf{2010 Mathematics Subject Classification:} Primary: 05C69;
05C30. Secondary: 05C31; 05C35.

\section{Introduction}\label{sec0}

 Throughout this paper, $G$ is a finite, undirected,
loopless graph without multiple edges, with vertex set $V(G)$ of
cardinality $n\left( G\right) $,
and edge set $E(G)$.   An edge $e\in E(G)$ connecting vertices $x$ and $y$ is
denoted by $xy$ or $yx$. In this case, the vertices $x$ and $y$ are said to
be \textit{adjacent}. A subset of $V(G)$ consisting of pairwise non-adjacent vertices
is called an \emph{independent} set. Denote  $\mathrm{Ind}(G)$ by  the family of all the independent sets of $G$.  An independent set is \textit{maximal} if it
cannot be extended by adding more vertices. Among all independent sets, one
with the largest cardinality is called a \emph{maximum} independent\textit{\ }set,
and its size is denoted $\alpha (G)$, known as the \textit{independence
number} of $G$.

A graph is \textit{well-covered} if all of its maximal independent sets have
the same cardinality \cite{Plummer1970, Plum93}. The class of well-covered
graphs contains all complete graphs $K_n$ and all complete bipartite graphs of the
form $K_{n,n}$. The only cycles which are well-covered are $%
C_{3},C_{4},C_{5} $, and $C_{7}$. Characterizing well-covered graphs is
known to be a difficult problem, and much of the existing literature has
focused on specific subclasses of well-covered graphs (see the survey in 
\cite{Plum93}). In the context of classifying well-covered graphs, Staples, in her thesis \cite{StaplesThesis}, introduced the class of graphs  belonging to  $\mathbf{W}_p$, which is defined as follows. 
  \begin{definition} 
  For a positive integer $p$, a graph $G$ is said to belong to the   $%
\mathbf{W}_{p}$ class if $n(G)\geq p$ and, for every collection of $p$ pairwise
disjoint independent sets $A_{1},\ldots ,A_{p}$ in $G$, there exist $p$
pairwise disjoint maximum independent sets $S_{1},\ldots ,S_{p}$ such that $%
A_{i}\subseteq S_{i}$ for all $1\leq i\leq p$.  
\end{definition}  

Furthermore, the classes $\mathbf{W}_{p}$ form a descending
chain: 
\begin{equation*}
	\mathbf{W}_{1}\supseteq \mathbf{W}_{2} \supseteq
	\cdots \supseteq \mathbf{W}_{p}\supseteq \cdots .
\end{equation*}%
Several constructions of $\mathbf{W}_{p}$ graphs are presented in detail in 
\cite{Favaron1982,   Pinter1991, StaplesThesis,Staples, TV92}.
It follows immediately that a graph with at least
one vertex belongs to the  $\mathbf{W}_{1}$ class if and only if it is
well-covered. Moreover, a graph is in $\mathbf{W}_{2}$ if and only if it is
a \textit{1-well-covered graph} without isolated vertices; that is, it is
well-covered, and the deletion of any vertex results in a graph that remains
well-covered \cite{StaplesThesis, Staples, Pinter2}. All complete graphs are also in $%
\mathbf{W}_{2}$, but no complete bipartite graphs (except $K_{1,1}$) are in $%
\mathbf{W}_{2}$. The cycles $C_{3}$ and $C_{5}$ are the only cycles in $%
\mathbf{W}_{2}$.

Let $S$ be a subset of the vertices of a graph $G$. The subgraph of $G$ \textit{induced} by $S$ is denoted $G[S]$, and the induced subgraph on the complement of $S$ is written $G-S$.  The {\it neighborhood} of $S$ is defined as
\[
N_{G}(S) = \{ v \in V(G) -  S \mid \text{$uv \in E(G)$ for some $u \in S$} \},
\]
and its \textit{closed neighborhood}  is $N_{G}[S] = S \cup N_{G}(S)$. The localization  of $G$ with respect to $S$ is the graph $G_S = G - N_G[S]$. For a singleton set $S = \{v\}$, we simplify the notation by writing $N_G(v)$, $N_G[v]$, $G-v$, and $G_v$, respectively. The  \textit{degree}  of a vertex $v$, denoted $\deg_G(v)$, is the cardinality of $N_G(v)$; a vertex of degree zero is called \textit{isolated}.

  For an edge $ab$ of $G$, let $G_{ab}$ denote the induced subgraph $G - (N_G(a) \cup N_G(b))$. We also define $G - ab$ as the graph obtained by deleting the edge $ab$ from $G$ while retaining all vertices and the remaining edges. Clearly, $\alpha(G) \leq \alpha(G-ab) \leq \alpha(G)+1$. An edge $ab$ of $G$ is called \textit{critical} if $\alpha(G-ab) > \alpha(G)$, equivalently, if $\alpha(G_{ab}) = \alpha(G)+1$.   A graph $G$ is said to be \textit{$\alpha$-critical} if every edge of $G$ is critical. It is clear that all odd cycles, as well as all complete graphs, are $\alpha$-critical.  This concept appears to have been first formulated and studied by Erd\"os and Gallai  \cite{EG61}. However, a structural characterization of $\alpha$-critical graphs remains unknown. In \cite{Plum67}, Plummer constructed an infinite family of such graphs, which in particular contains all $\alpha$-critical graphs with fewer than eight vertices. Some related results on $\alpha$-critical graphs have also been studied in \cite{BHP67, B82, Plummer1970}.

  In  \cite[Pages 20-21]{Plum93}, Plummer posed several open questions, including one concerning
the characterization of graphs that are both $\alpha$-critical and belong to
the   $\mathbf{W}_1$ or  $\mathbf{W}_2$ class.  This problem remains unresolved.  The aim of the present work is to study this problem in a more general setting for $\alpha$-critical graphs belonging to the $\mathbf{W}_p$ class with $p \ge 1$. The main result of this paper provides a characterization of a sufficient condition for a graph to be both $\alpha$-critical and in the $\mathbf{W}_p$ class. Moreover,  in the case where $G$ is locally triangle-free, we establish an equivalent characterization of this class of graphs.

  The paper is organized as follows. In Section \ref{sec1}, we begin by
  recalling some basic notations together with fundamental properties of the $%
  \mathbf{W}_{p}$ class. Section \ref{sec2} deals with new characterizations
  of $\mathbf{W}_{p}$ graphs. The purpose of Section \ref{sec3} is to
  characterize $\alpha $-critical graphs belonging to $\mathbf{W}_{p}$
  classes. In particular, we provide a characterization for the class of
  locally triangle-free graphs.

\section{Structural properties} 
\label{sec1}

The following lemma  provides a necessary and sufficient condition for a graph to be well-covered, a result established in \cite[Theorem 5.3]{Plum93}, \cite[Lemma 1]{FHN93}, and \cite[Lemma 4.1]{HT16}.

\begin{lemma}\label{CP88a}  Let $G$ be a graph with $\alpha(G)>1$. Then 
  $G$ is a well-covered
graph if and only if  $G_{v}$ is also well-covered and $\alpha
(G_{v})=\alpha (G)-1$ for all $v\in V(G)$.
\end{lemma}

\begin{lemma}
 {\rm (\cite[Lemma 1]{FHN93})} \label{CP88b} If $G$ is a well-covered graph,
and $S$ is an independent set of $G$ such that $|S|<\alpha (G)$, then $G_{S}$
is also well-covered and $\alpha (G)=\alpha (G_{S})+|S|$.
\end{lemma}

 We shall invoke the following lemma at several points in this paper. 
\begin{lemma}
	\label{commutativity} If $S,T\subseteq V(G) $ such that $%
	N_{G}[S]\cap T=\emptyset $ and $N_{G}[T]\cap S=\emptyset $,
	then 
	\begin{equation*}
		(G_{S})_{T}=G_{S\cup T}=(G_{T})_{S}.
	\end{equation*}
\end{lemma}

\begin{proof}
	From the assumption, we obtain  the symmetric containments $S \subseteq V(G_{T})$ and $T \subseteq V(G_{S})$. Hence, the order of localization is commutative; that is,  
	\begin{eqnarray*}
		(G_{S})_{T} &=& G_{S} - N_G[T] = \left(G - N_G[S]\right) - N_G[T] = G - N_G[S \cup T]\\
		&=&   \left(G - N_G[T]\right) - N_G[S] = G_{T} - N_G[S] = (G_{T})_{S},
	\end{eqnarray*}     
	which completes the proof. 
\end{proof}

 A vertex $v \in V(G)$ is called a \textit{shedding} if, for every independent set $S$ of $G_v$, there exists a vertex $u \in N_G(v)$ such that $S \cup \{u\}$ is also  an  independent set \cite{W09}. We denote by $\mathrm{Shed}(G)$ the set of all shedding vertices of $G$. It is immediate that no isolated vertex can be a shedding vertex. Conversely, every vertex of $G$  with  degree $n(G)-1$ is necessarily a shedding vertex of $G$.

   As stated previously, graphs with at least two vertices in the class $\mathbf{W}_2$, equivalently in the class of 1-well-covered graphs, are precisely those graphs $G$ that are well-covered and for which $G-v$ is also well-covered with $\alpha(G-v) = \alpha(G)$. Note that if $v$ is not an isolated vertex of well-covered graph $G$, then $\alpha(G-v) = \alpha(G)$. Thus, in order to characterize this class, one must determine the criterion under which $G-v$ is well-covered. In \cite[Lemma 2]{FHN93}, Finbow, Hartnell, and Nowakowski established a necessary and sufficient condition for determining when $G-v$ is well-covered. Later, Castrillón, Cruz, and Reyes \cite[Lemma 2]{CCR16} provided an additional characterization in terms of shedding vertices, stated as follows: 

\begin{theorem} \label{thm_G-v}
 Let $G$ be a well-covered graph. Given a non-isolated vertex $%
	v\in V(G)$, the following conditions are equivalent:
	
	\begin{enumerate} 
		\item[(a)] $G-v$ is well-covered;
		
		\item[(b)] $\left\vert N_{G}(v)-N_{G}(S)\right\vert \geq 1$ for every independent
		set $S$ of $G_{v}$;
		
		\item[(c)] there is no independent set $S\subseteq V(G_{v})$ such that $v$ is
		isolated in $G_{S}$;
		
		\item[(d)] $v$ is a shedding vertex.
	\end{enumerate}
\end{theorem}

The differential of a set $A\subseteq V(G)$ is $\partial (A)=\left\vert
N_G(A)-A\right\vert -|A|$ (\cite{BF12}). Clearly, if $S$ is independent, then $%
\partial (S)=\left\vert N_G(S)\right\vert -|S|$.  Analogous to the case of well-covered graphs, the second and third authors have derived the following characterizations of graphs in $\mathbf{W}_2$ as follows:   

\begin{theorem}
	 {\rm (\cite[Theorem 3.9]{LM2017})}  \label{th12}Let $G$\ be a well-covered graph without isolated
	vertices. Then the following assertions are equivalent:
	
	\begin{enumerate}
		\item[(a)] $G$\emph{\ }belongs to the  $\mathbf{W}_{2}$ class;
		
		\item[(b)] the differential function is monotonic over $\mathrm{Ind}(G)$, i.e., if $A\subseteq B\in \mathrm{Ind}(G) $, then $%
		\partial(A) \leq \partial(B) $;
		
		\item[(c)] $\mathrm{Shed}(G) =V(G) $;
		
		\item[(d)] no independent set $S$ leaves an isolated vertex in $G-N_{G}[S]$;
		
		\item[(e)] $G_v \in \mathbf{W}_{2}$ for every $v\in
		V(G) $.
	\end{enumerate}
\end{theorem}

In the general case,  Staples \cite[Lemma and Theorem 1]{Staples} identified several initial characterizations of the  $\mathbf{W}_{p}$ class, as follows.

\begin{theorem}
\label{lem_equiv}   Let $p\geq2$. Then

\begin{enumerate}
\item[(a)] $G\in \mathbf{W}_{p}$ if and only if $G-v\in \mathbf{W}_{p-1}$ and $%
\alpha (G)=\alpha (G-v)$ for all $v\in V(G)$.

\item[(b)] $G\in \mathbf{W}_{p}$ if and only if for every set $A\subseteq V(G)$
with $\left\vert A\right\vert =p-1$, the graph $G-A$ is well-covered with $%
\alpha (G-A)=\alpha (G)$.
\end{enumerate}
\end{theorem}

Furthermore, in \cite[Constructions 1–4]{Staples}, Staples presented several constructions of infinite families of $\mathbf{W}_p$ graphs that admit independent sets of arbitrarily large cardinality. The following property follows directly from the definition and will be used repeatedly throughout this paper.

\begin{lemma} {\rm (\cite[Theorems 3 and 4]{Staples})}
\label{lem_key} Let $p \geq 2$, and suppose that $G$ is in  $\mathbf{W}_{p}$
class. Then the following properties hold:

\begin{enumerate}
\item[(a)]  $n(G)\geq p\cdot \alpha (G)$. In particular,
equality holds, i.e., $n(G)=p\cdot \alpha (G)$, if and only if $G$ is the
disjoint union of $\alpha (G)$ complete graphs, each on $p$ vertices.

\item[(b)]   if $G$ is connected and non-complete, then
every vertex in $G$ has degree at least $p$.
\end{enumerate}
\end{lemma}

\begin{theorem}  {\rm (\cite[Theorem 2.4]{DLMP})}  \label{th1} Let $G$ be   a graph  without isolated vertices in    $\mathbf{W}_{p}$ class,  
	and $A$ be a non-maximum independent set in $G$.
	Then the following assertions are true.
	\begin{enumerate}
		\item[(a)] There are at least $p$ pairwise disjoint independent sets $%
		B_{1},B_{2},\ldots,B_{p}$ such that $A\cup B_{i}$ is maximum independent set of $G$ 
		and $A\cap B_{i}=\emptyset $ for each $1\leq i\leq p$.
		\item[(b)] If $p\geq 2$, then there are at least $p-1$ pairwise disjoint maximum
		independent sets $S_{1},S_{2},\ldots,S_{p-1}$ such that $A\cap S_{i}=\emptyset $
		for each $1\leq i\leq p-1$.
	\end{enumerate}
\end{theorem}

The following lemma states that a graph $G$ in the class $\mathbf{W}_p$ is preserved under taking induced subgraphs on the complements of closed neighborhoods. 

\begin{lemma} {\rm (\cite[Lemma 2.7]{HLMP2024}) }
\label{key2} Let $G$ be a $\mathbf{W}_{p}$ graph. The following assertions
are true:

\begin{enumerate}
\item[(a)] if $\alpha (G)>1$, then $G_{x}\in 
\mathbf{W}_{p}$ for every $x\in V(G)$;

\item[(b)]  if $S$ is an independent set of $G$
such that $\left\vert S\right\vert <\alpha (G)$, then $G_{S}\in \mathbf{W}%
_{p}$. In particular, if $p > 1$, then $G_S$ has no isolated vertices.
\end{enumerate}
\end{lemma}

\section{Characterizing $\mathbf{W}_{p}$ graphs}
\label{sec2}  
 
 For $p \ge 1$, every graph belonging to the  $\mathbf{W}_{p}$ class necessarily contains at least $p$ vertices. Moreover, by Lemma \ref{lem_key}~(b), such a graph has no isolated vertices whenever $p \ge 2$. In addition, each connected component of a $\mathbf{W}_{p}$ graph is itself a member of $\mathbf{W}_{p}$, as formalized in the following lemma: 
\begin{lemma}
 {\rm(\cite[Theorem 2.6]{HLMP2024})} \label{disconnected} A graph is in $%
\mathbf{W}_{p}$  if and only if each of its connected components is also $%
\mathbf{W}_{p}$.
\end{lemma}

\begin{theorem}
\label{mth1} Let $p\geq 1$ and $G$ be a graph with $\alpha (G)\geq 2$. Then $%
G\in \mathbf{W}_{p}$ if and only if $G_{x}\in \mathbf{W}_{p}$ and $\alpha
(G_{x})=\alpha (G)-1$ for every $x\in V(G)$.
\end{theorem}

\begin{proof}
For $p = 1$, the necessary condition of this theorem was established in
Lemma \ref{CP88a}, and the sufficient condition was also proved in \cite[%
Lemma 4.1]{HT16}. For $p = 2$, the necessary condition is shown in \cite[%
Theorem 5]{Pinter}, while the sufficient condition is proved in \cite[%
Theorem 3.9]{LM19}. Now we assume that $p\ge 2$.

\vskip0.5em \noindent ($\Longrightarrow $) Follows from Lemma \ref{key2}(b)
and Lemma \ref{CP88a}.

\vskip0.5em \noindent ($\Longleftarrow $) By Theorem \ref{lem_equiv}, we
need to prove that $G-v\in \mathbf{W}_{p-1}$ and $\alpha (G-v)=\alpha (G)$
for all $v\in V(G)$. Since $\mathbf{W}_{p}\subseteq \mathbf{W}_{1}$, 
by the assumption, $G_x$ is well-covered and $\alpha(G_x)=\alpha(G)-1$ for all $x\in V(G)$. Applying Lemma \ref{CP88a}, $G$ is a well-covered graph. Moreover, by
the definition of $\mathbf{W}_p$ graphs, $n(G_{x}) \geq p$.

\vskip0.5em \noindent \emph{Claim 1.} $G$ has no isolated vertices.

\vskip0.5em Suppose that $G$ has an isolated vertex, say $v$. Since $\alpha
(G)\geq 2$, there exists a vertex $x\in V(G)$ such that $\{x,v\}$ is an
independent set in $G$. It follows that $v$ is also an isolated vertex in
the graph $G_{x}$. Now, if $\alpha (G_{x})=1$, then $G_{x}$ must be a
complete graph. But since $v$ is an isolated vertex in $G_{x}$, the only
possibility is that $G_{x}$ consists of the vertex $v$, contradicting the
assumption that $n(G_{x})\geq 2$. Therefore, we must have $\alpha (G_{x})>1$%
. But this contradicts Lemma~\ref{key2}~(b), since $G_{x}\in \mathbf{W}_{p}$.

\vskip0.5em \noindent \emph{Claim 2.} $\alpha (G-v)=\alpha (G)$ for every $%
v\in V(G)$.

\vskip0.5em By \textit{Claim 1}, the vertex $v$ must be adjacent to some
vertex $w$ in $G$. Let $S$ be a maximal independent set in $G$ that contains 
$w$. Since $G$ is well-covered, we have $\left\vert S\right\vert =\alpha (G)$. Furthermore, since $S$ is entirely contained in $V(G-v)$, it follows that $%
\alpha (G-v)=\alpha (G)$.

\vskip0.5em \noindent \emph{Claim 3.} $G-v$ is in $\mathbf{W}_{p-1}$.

\vskip0.5em We prove this claim by induction on $n(G)+p$. By Lemma \ref%
{lem_key}~(a), $n(G_{x})\geq p\cdot \alpha (G_{x})$. Equivalently, $%
n(G)-\left\vert N_{G}[x]\right\vert \geq p\cdot\alpha(G) -p$. Thus, by \emph{Claim 1}, 
\begin{equation*}
n(G)+p \geq p\cdot\alpha(G) + \left\vert N_{G}[x]\right\vert \geq p\cdot
\alpha(G) +2.
\end{equation*}

If $n(G)+p=p\cdot \alpha (G)+2$, then $\left\vert N_{G}[x]\right\vert =2$, so $%
\deg _{G}(x)=1$. Let $y$ be the unique neighbor of $x$ in $G$. By the same
reasoning, $\left\vert N_{G}[y]\right\vert =2$, which implies that the edge $%
xy$ defines a connected component in $G$. Now, consider any vertex $z\in
V(G_{x})$. Then the graph $G_{z}$ contains a connected component isomorphic
to $K_{2}$, namely the edge $xy$. By assumption, $G_{z}\in \mathbf{W}_{p}$,
and hence, by Lemma \ref{disconnected}, $K_{2}\in \mathbf{W}_{p}$, which
forces $p\leq 2$.  By \cite[%
Theorem 3.9]{LM19},  $G\in \mathbf{W}_{p}$. Consequently, $%
G-v\in \mathbf{W}_{p-1}$ by Theorem \ref{lem_equiv}~(a). 

We now assume that $n(G)+p>p\cdot \alpha(G)+2$. For every $x\in V(G-v)$, we
claim that $(G-v)_{x}\in \mathbf{W}_{p-1}$ and $\alpha ((G-v)_{x})=\alpha
(G-v)-1$. To prove this, we divide the argument into the following two cases.

\vskip0.5em \noindent \emph{Case 1.} $x$ is not adjacent to $v$ in $G$.

\vskip0.5em In this case, we have 
\begin{equation*}
(G-v)_{x}=(G-v)-N_{G}[x]=(G-N_{G}[x])-v=G_{x}-v.
\end{equation*}%
By the assumption, $G_{x}\in \mathbf{W}_{p}$. Applying Theorem \ref%
{lem_equiv}~(a), we conclude that $G_{x}-v$ belongs to $\mathbf{W}_{p-1}$ and
satisfies $\alpha (G_{x}-v)=\alpha (G_{x})$. Moreover, together with \emph{%
Claim 2}, we obtain 
\begin{equation*}
\alpha ((G-v)_{x})=\alpha (G_{x}-v)=\alpha (G_{x})=\alpha (G)-1=\alpha
(G-v)-1.
\end{equation*}%
\vskip0.5em \noindent \emph{Case 2.} $x$ is adjacent to $v$ in $G$.

\vskip0.5em In this case, we have 
\begin{equation*}
(G-v)_{x}=(G-v)-N_{G}[x]=G-N_{G}[x]=G_{x}.
\end{equation*}%
By the asumption, we obtain that $(G-v)_{x}\in \mathbf{W}_{p}\subseteq 
\mathbf{W}_{p-1}$ and by \emph{Claim 2} again, 
\begin{equation*}
\alpha ((G-v)_{x})=\alpha (G_{x})=\alpha (G)-1=\alpha (G-v)-1.
\end{equation*}%
\vskip0.5em From \emph{Case 1} and \emph{Case 2}, we obtain that $%
(G-v)_{x}\in \mathbf{W}_{p-1}$ and $\alpha ((G-v)_{x})=\alpha (G-v)-1$ for
every $x\in V(G-v)$. Since $n(G-v)+(p-1)<n(G)+p$, by the induction
hypothesis, it follows that $G-v$ belongs to $\mathbf{W}_{p-1}$, as claimed.
\end{proof}

\begin{theorem} \label{mainthm}
Let $p\ge 1$ and $G\in \mathbf{W}_{p}$. For a non-isolated vertex $v$ of $G$%
, the following conditions are equivalent:

\begin{enumerate}
\item[(a)]  $G-v$ is in $\mathbf{W}_{p}$;

\item[(b)]  $\left\vert N_{G}(v)-N_{G}(S)\right\vert \geq p$ for every
independent set $S$ of $G_{v}$;

\item[(c)]  there is no independent set $S\in V(G_{v})$ such that $\left\vert
N_{G_{S}}(v)\right\vert \leq p-1$.
 
 \end{enumerate}
\end{theorem}

\begin{proof} By   assumption, $G$ is well-covered and $v$ is an isolated vertex, so $\alpha(G-v)=\alpha(G)$.
\vskip0.5em    
(b) $\Longleftrightarrow $ (c):  Let $S$ be an independent set of $G_{v}$. The
claim is clear, because 
\begin{equation*}
\left\vert N_{G_{S}}(v)\right\vert =\left\vert N_{G}(v)-N_{G}(S)\right\vert .
\end{equation*}

\vskip0.5em  (a) $\Longrightarrow $ (b): Suppose there exists an independent
set $S$ in $G_{v}$ such that 
\begin{equation*}
\left\vert N_{G}(v)-N_{G}(S)\right\vert \leq p-1.
\end{equation*}%
Set $\left\vert N_{G}(v)-N_{G}(S)\right\vert =t$. In other words, $%
N_{G}(v)-N_{G}(S)=\left\{ u_{1},...,u_{t}\right\} $, where $1\leq t\leq p-1$%
. Each of vertices $u_{i}\in N_{G}(v)-N_{G}(S),1\leq i\leq t$ forms an
independent set $\{u_{i}\}$ in $G$. Clearly, these sets and $S$ are disjoint
in $G-v$. Hence, by definition of a $\mathbf{W}_{p}$ graph, there exists a
family of pairwise disjoint maximum independent sets $S_{1},\ldots
,S_{t},S_{t+1}$ in $G-v$ such that $u_{i}\in S_{i}$ for $1\leq i\leq t$, and 
$S\subseteq S_{t+1}$. Therefore, $u_{1},\ldots ,u_{t},v\notin S_{t+1}$ and 
\begin{equation*}
N_{G}(v)=\{u_{1},\ldots ,u_{t}\}\cup (N_{G}(S)\cap N_{G}(v)).
\end{equation*}

Since $S_{t+1}$ is an independent set containing $S$, we know that $%
N_{G}(S)\cap S_{t+1}=\emptyset $. Thus, 
\begin{equation*}
N_{G}(v)\cap S_{t+1}\subseteq (N_{G}(v)\cap N_{G}(S))\cap S_{t+1}\subseteq
N_{G}(S)\cap S_{t+1}=\emptyset .
\end{equation*}

Hence, $S_{t+1}\cup \{v\}$ is an independent set of $G$. On the other hand,
each $S_{i},1\leq i\leq t+1$ has size $\alpha (G-v)=\alpha (G)$.
Consequently, $S_{t+1}\cup \{v\}$ would be an independent set in $G$ of size 
$\alpha (G)+1$, contradicting the definition of $\alpha (G)$.

\vskip0.5em (b) $\Longrightarrow $ (a): 
If $p=1$, the   assertion follows directly from Theorem \ref{thm_G-v}. We now consider the case   $p\ge 2$. 
 First, taking $S=\emptyset $ gives $%
N_{G}(S)=\emptyset $. Therefore, for each vertex $v\in V(G)$, we have 
\begin{equation*}
\left\vert N_{G}(v)\right\vert =\left\vert N_{G}(v)-N_{G}(S)\right\vert \geq
p.
\end{equation*}

Further, we proceed by the induction on $\alpha (G)$. If $\alpha (G)=1$, then $G$ is a complete graph on at least $p+1$ vertices.
Consequently, $G-v$ is a complete graph on at least $p$ vertices, and thus $%
G-v\in \mathbf{W}_{p}$.

Assume $\alpha (G)\geq 2$. By Theorem \ref{mth1}, it is sufficient to prove
that $(G-v)_{x}\in \mathbf{W}_{p}$, and $\alpha ((G-v)_{x})=\alpha (G-v)-1$
for each $x\in V(G-v)$.

In what follows, we distinguish between   two following cases: \vskip0.5em
\noindent \emph{Case 1.} Assume that $x$ is adjacent to $v$ in $G$. \vskip%
0.5em \noindent In this situation, we have 
\begin{equation*}
(G-v)_{x}=(G-v)-N_{G}[x]=G-N_{G}[x]=G_{x}.
\end{equation*}%
On the other hand, $\alpha (G_{x})=\alpha (G)-1$, because $G$ is
well-covered. Now, by the assumption and Lemma~\ref{key2}, $G_{x}\in \mathbf{W}_{p}$ and $%
\alpha (G_{x})=\alpha (G)-1$, i.e., $(G-v)_{x}\in \mathbf{W}_{p}$, and 
\begin{equation*}
\alpha ((G-v)_{x})=\alpha (G_{x})=\alpha (G)-1=\alpha (G-v)-1,
\end{equation*}%
as claimed. \vskip0.5em \noindent \emph{Case 2.} Assume that $x$ is not
adjacent to $v$ in $G$. \vskip0.5em \noindent In this situation,   $v\in V(G_x)$. Then  we have 
\begin{equation*}
(G-v)_{x}=G-v-N_{G}[x]=G-N_{G}[x]-v=G_{x}-v.
\end{equation*}%
  By assumption, $G$ is well-covered, and since $v$ is not an isolated vertex of $G$, it follows that $\alpha(G-v)=\alpha(G)$. Furthermore, by Theorem \ref{mth1}, we have $G_x \in \mathbf{W}_p$, $\alpha(G_x)=\alpha(G)-1$, and, by Lemma \ref{key2}(b), $v$ is not an isolated vertex of $G_x$. In addition, Theorem \ref{lem_equiv}(a) ensures that $G_x-v \in \mathbf{W}_{p-1}$ and $\alpha(G_x-v)=\alpha(G_x)$. Therefore, we conclude that 
\begin{equation*}
\alpha ((G-v)_{x})=\alpha \left( G_{x}-v\right) =\alpha \left( G_{x}\right)
=\alpha (G)-1=\alpha (G-v)-1
\end{equation*}%
for each $x\in V(G-v)$.
 
 Let $S$ be an arbitrary independent set of $(G_{x})_{v}$. By Lemma \ref{commutativity}, we have $(G_{x})_{v}=G_{\{x,v\}}$. Hence, $S \cup \{x\}$ forms an independent set in $G_{v}$, and  
\begin{equation*}
\left\vert N_{G_{x}}(v)-N_{G_{x}}(S)\right\vert \geq \left\vert
N_{G}(v)-N_{G}(S\cup \{x\})\right\vert \geq p.
\end{equation*}%
By the induction hypothesis, $G_{x}-v\in \mathbf{W}_{p}$. Hence, $%
(G-v)_{x}\in \mathbf{W}_{p}$ and $\alpha ((G-v)_{x})=\alpha (G-v)-1$ for
each $x\in V(G-v)$, as claimed.
\end{proof}

\section{A characterization of $\protect\alpha$-critical graphs in $\mathbf{W%
}_p$ class} \label{sec3}

  For any edge $ab$ of a graph $G$, recall that $G$ is called \emph{$\alpha$-critical} if $\alpha(G-ab) > \alpha(G)$, equivalently, if $\alpha(G-ab) = \alpha(G)+1$ for every edge $ab \in E(G)$. 
   In \cite[Theorem 3.10]%
  {StaplesThesis},  Staples  proved that
every triangle-free graph  in  $\mathbf{W}_{2}$ is necessarily $\alpha $-critical. The
purpose of this section is to address a question posed by Plummer in \cite[%
Problem 9(b)]{Plum93}, where he raised an open problem concerning the
characterization of $\alpha $-critical graphs within the  $\mathbf{W}%
_{2}$ class. Furthermore, we provide a general characterization of locally triangle-free  $\alpha $%
-critical graphs in the  $\mathbf{W}_{p}$ class.

\begin{lemma}
\label{disconnected1} Let $G_{1},\ldots ,G_{k}$ be all the connected
components of $G$. Then $G$ is $\alpha $-critical if and only if all $G_i$
are also $\alpha $-critical for all $1\le i\le k$.
\end{lemma}

\begin{lemma}
{\rm(\cite[Lemma 4.1]{HT16})} \label{HoangTrung_lem1} If $G_{ab}$ is
well-covered graph and $\alpha (G_{ab})=\alpha (G)-1$ for every $ab\in E(G)$, then $G$ is well-covered.
\end{lemma}

The following lemma was originally established in \cite[Theorem 4.7(d)]{JV21} with a proof formulated in the language of commutative algebra. In what follows, we present a simpler proof relying solely on combinatorial arguments.  

\begin{lemma}
\label{Gab_critical} $G$ is $\alpha $-critical if and only if $\alpha
(G_{ab})=\alpha (G)-1$ for each $ab\in E(G)$.
\end{lemma}

\begin{proof}
($\Longrightarrow $) For each edge $ab\in E(G)$, we have $\alpha
(G_{ab})\leq \alpha (G)-1$. Since $G$ is $\alpha $-critical,  $\alpha
(G-ab)=\alpha (G)+1$. Therefore, there exists an\textbf{\ }independent set
of $G-ab$ that contains both $a$ and $b$, say $S$, such that  $\left\vert
S\right\vert =\alpha (G-ab)=\alpha (G)+1$.

Define $S^{\prime }=S-\{a,b\}$. Then $S^{\prime }$ is an independent set in $%
G_{ab}$ and $\left\vert S^{\prime }\right\vert \leq \alpha (G_{ab})$.
Hence, 
\begin{equation*}
\left\vert S^{\prime }\right\vert =\left\vert S\right\vert -2=\alpha
(G)-1\leq \alpha (G_{ab})\leq \alpha (G)-1.
\end{equation*}
Therefore, $\alpha (G_{ab})=\alpha (G)-1$.

\vskip0.5em \noindent ($\Longleftarrow $) For each $ab\in E(G)$, by the assumption, $\alpha (G_{ab})=\alpha (G)-1$. Let $S$ be a maximum
independent set in $G_{ab}$, so $\left\vert S\right\vert =\alpha (G)-1$.
Then $S\cup \{a,b\}$ is an independent set in $G-ab$, so $\left\vert S\cup
\{a,b\}\right\vert \leq \alpha (G-ab)$. Hence, $\alpha(G)< \alpha(G-ab).$ 
 \end{proof}

The following theorem was proved in the case $p=2$ and $G$ is triangle-free
graph in \cite[Lemma 4.2]{HT16}.

\begin{theorem}
\label{mthm1} Let $p\geq 2$ and $G$ be a graph with $\alpha(G)>1$. If $%
G_{ab}\in \mathbf{W}_{p-1}$ and $\alpha (G_{ab})=\alpha (G)-1$ for every $%
ab\in E(G)$, then $G\in \mathbf{W}_{p}$ and $\alpha $-critical.
\end{theorem}

\begin{proof}
By Lemmas \ref{disconnected} and \ref{disconnected1}, it is enough to prove
the theorem for connected graphs only. Now we may assume that $G$ is
connected.

By Lemma \ref{Gab_critical}, $G$ is $\alpha $-critical. Moreover, by
definition, $n(G_{ab}) \geq p-1$, since $G_{ab}\in \mathbf{W}_{p-1}$.
Therefore, $\alpha (G_{ab})\geq 1$, so $\alpha (G)>1$. Since $%
G_{ab}\in \mathbf{W}_{p-1}\subseteq \mathbf{W}_{1}$,    $G_{ab}$ is
well-covered and $\alpha (G_{ab})=\alpha (G)-1$ for all $ab\in E(G)$. By Lemma \ref%
{HoangTrung_lem1}, $G$ is also well-covered. Lemma \ref{CP88a} implies that $%
\alpha (G_{x})=\alpha (G)-1$ for all $x\in V(G)$.

  In order to establish that $G \in \mathbf{W}_p$, it is sufficient, by Theorem \ref{mth1}, to verify that $G_x \in \mathbf{W}_p$ for every vertex $x \in V(G)$. We shall prove this claim  by induction on   $\alpha(G)$.  
 
  Suppose first that $\alpha(G)=2$. Then $\alpha(G_{ab})=1$. Since $G$ is well-covered, we have $\alpha(G_x)=\alpha(G)-1=1$ for all $x\in V(G)$, which implies that each $G_x$ is a complete graph. Because $G$ is connected, there exist vertices $y \in N_G(x)$ and $v \in V(G_x)$ such that $vy \in E(G)$. Clearly, $G_{xy}=G_x - N_G(y)$ and $v \in N_G(y) \cap V(G_x)$. By assumption, $G_{xy} \in \mathbf{W}_{p-1}$, and hence
  $n(G_{xy}) \geq p-1$.   
  It follows that 
\begin{equation*}
n(G_{x}) =n(G_{xy}) +\left\vert N_{G}(y)\cap V(G_{x})\right\vert \geq
p-1+1=p.
\end{equation*}
  Moreover, since $\alpha(G_x)=\alpha(G)-1=1$, the graph $G_x$ is complete of order at least $p$. Therefore, $G_x \in \mathbf{W}_p$.

Now, assume that $\alpha (G)\ge 3$. We claim that $G_{x}$ has no isolated
vertices. Indeed, assume $v$ is an isolated vertex of $G_{x}$. Since $G$ is
connected, there is a vertex $w\in N_{G}(x)$ such that $vw\in E(G)$. Then $%
G_{x}=G_{vw}\cup \{v\}$. Then $\alpha (G_{x})=\alpha (G_{xw})+1=\alpha
(G)-1+1=\alpha (G)$, a contradiction.

Let $ab$ be an arbitrary edge of $G_{x}$. By Lemma \ref{commutativity}, we
know that $(G_{x})_{ab}=(G_{ab})_{x}$. According to Theorem \ref{mth1}, $%
(G_{ab})_{x}$ is in $\mathbf{W}_{p-1}$ and $\alpha ((G_{ab})_{x})=\alpha
(G_{ab})-1$. Therefore, $(G_{x})_{ab}\in \mathbf{W}_{p-1}$ and moreover, $$\alpha ((G_{x})_{ab})=\alpha((G_{ab})_x)=\alpha(G_{ab})-1 = \alpha(G)-2= \alpha (G_{x})-1.$$ 
Therefore, $G_x$ is $\alpha$-critical by Lemma \ref{Gab_critical}.  By the induction hypothesis, $G_{x}\in \mathbf{W}_{p}$ for
all $x\in V(G)$.  
\end{proof}

 A graph $G$ is said to be \emph{locally triangle-free} if $G_{x}$ is triangle-free for every $x \in V(G)$. Note that a locally triangle-free graph may still contain a triangle as a subgraph, whereas every triangle-free graph is necessarily locally triangle-free.

\begin{corollary}
Let $p\geq 2$ and $G$ be a locally triangle-free graph with $\alpha(G)>1$.
Then $G_{ab}\in \mathbf{W}_{p-1}$ and $\alpha (G_{ab})=\alpha (G)-1$ for
every $ab\in E(G)$  if and only if $G\in \mathbf{W}_{p}$ and $\alpha $%
-critical.
\end{corollary}

\begin{proof}
($\Longrightarrow$) follows from Theorem \ref{mthm1}.

($\Longleftarrow $) Since $G\in \mathbf{W}_{p}\subseteq \mathbf{W}_{2}$,
according to Lemma \ref{Gab_critical}, $\alpha (G_{ab})=\alpha (G)-1$ for
all $ab\in E(G)$. Therefore, it remains to show that $G_{ab}$ is in $\mathbf{%
W}_{p-1}$ for all $ab\in E(G)$. We prove this by induction on $\alpha (G)$.
If $\alpha (G)=2$, then since $G\in \mathbf{W}_{p}\subseteq \mathbf{W}_{2}$,
it follows from \cite[Proposition 1.7]{HT16} that $G\cong C_{n}^{c}$ for
some $n\geq 4$. Because $G$ is $\alpha $-critical, we must have $n=5$.
Hence, $G\cong C_{5}$, and in this case, the statement clearly holds.

Suppose that $\alpha (G)>2$. For all $x\in V(G_{ab})$, we have 
\begin{equation*}
(G_{ab})_{x}=G_{ab}-N_{G}[x]=G-N_{G}[\{a,b\}]-N_{G}[x]=G-N_{G}[x]-N_{G}[%
\{a,b\}]=(G_{x})_{ab}
\end{equation*}%
Since $G\in \mathbf{W}_{p}$ and $\alpha (G)>1$, by Lemma \ref{key2}(a), $%
G_{x}\in \mathbf{W}_{p}$. Moreover, by the assumption, $G_{x}$ is
triangle-free and thus $G_{x}$ is $\alpha $-critical by \cite[Theorem 3.10]%
{StaplesThesis}. By the induction, $(G_{x})_{ab}\in \mathbf{W}_{p-1}$ and $\alpha
((G_{x})_{ab})=\alpha (G_{x})-1$. Therefore, $(G_{ab})_{x}\in \mathbf{W}%
_{p-1}$ and 
\begin{equation*}
\alpha ((G_{ab})_{x})=\alpha ((G_{x})_{ab})=\alpha (G_{x})-1=\alpha
(G)-2=\alpha (G_{ab})-1.
\end{equation*}%
By Theorem \ref{mth1}, $G_{ab}$ is in $\mathbf{W}_{p-1}$.
\end{proof}

\begin{corollary} Let $p\geq 2$ and $G$ be a   triangle-free graph with $\alpha(G)>1$.
	Then $G_{ab}\in \mathbf{W}_{p-1}$ and $\alpha (G_{ab})=\alpha (G)-1$ for
	every $ab\in E(G)$ if and only if $G\in \mathbf{W}_{p}$. 
\end{corollary}

\begin{example} 
 \begin{enumerate}
\item[(a)]   Figure 1 in \cite{HT18} presents several graphs that are both locally triangle-free in  $\mathbf{W}_2$ and $\alpha$-critical. 
 \item[(b)]  For every  $p \geq 1$, the graph  $G \circ K_p$ belongs  in the   $\mathbf{W}_p$ class, but  it does not    $\alpha$-critical whenever $n(G)>1$. 
  \item[(c)]  For $n_1,n_2,m_1,m_2\ge p$, $(K_{n_1} \cup K_{n_2})+ (K_{m_1}\cup K_{m_2})$ belongs to $\mathbf{W}_p$  class  but not $\alpha$-critical. In particular, it  is locally triangle-free when $n_1,n_2,m_1,m_2\le 2.$   
 \end{enumerate}
 \end{example}

\section*{Conclusion}

The characterizations obtained for the  $\mathbf{W}_{p}$ class naturally
suggest the following.  
\vskip0.5em 
\noindent  
\textbf{Question.} Characterize $\alpha $-critical graphs belonging to the
 $\mathbf{W}_{p}$ class  for $p\geq 1$.
\section*{Acknowledgment}

Do Trong Hoang is also partially supported by NAFOSTED (Vietnam) under the
grant number 101.04-2024.07.

\section*{Declarations}

\noindent \textbf{Conflict of interest/Competing interests} \newline
The authors declare that they have no competing interests\newline
\noindent \textbf{Ethical approval and consent to participate}\newline
\noindent Not applicable. \newline
\noindent \textbf{Consent for publication}\newline
\noindent Not applicable. \newline
\noindent \textbf{Availability of data, code and materials}\newline
\noindent Data sharing not applicable to this work as no data sets were
generated or analyzed during the current study. \newline
\noindent \textbf{Authors' contribution}\newline
All authors have contributed equally to this work.\newline


\begin{thebibliography}{99}
	
	
	\bibitem{BHP67}  L.  W. Beineke, F. Harary,   M. D. Plummer, \emph{On the critical lines of a graph}, 
	Pacific Journal of Mathematics, 
	{\bf 22} (2)  (1967), 205--212. 
	
	
\bibitem{B82} C. Berge, \emph{Some common properties for regularizable
graphs, edge-critical graphs and B-graphs}, Annals of Discrete Mathematics, 
\textbf{12} (1982), 31--44.


\bibitem{BF12} S. Bermudo,    H. Fernau, \emph{Lower bounds on the differential
	of a graph}, Discrete Mathematics,  \textbf{312} (2012), 3236--3250.


\bibitem{CCR16}  I. D. Castrill\'on, R. Cruz,    E. Reyes, {\it On well-covered, vertex decomposable and Cohen–Macaulay graphs}, Electronic Journal of Combinatorics, {\bf  23} (2) (2016), 17 pp. 


\bibitem{EG61} P. Erd\"os,   T. Gallai,  \emph{On the Minimal Number of Vertices Representing 
the Edges of a Graph}, Publications of the Mathematical Institute of the Hungarian Academy of Sciences,  
 {\bf 6} (1961), 181--203. 


\bibitem{Favaron1982} O. Favaron, \emph{Very well-covered graphs}, Discrete
Mathematics, \textbf{42} (1982) 177--187.

\bibitem{FHN93} A. Finbow, B. Hartnell,    R. Nowakowski, \emph{A
characterization of well-covered graphs of girth 5 or greater}, Journal of
Combinatorial Theory. Series B, \textbf{57} (1993) 44--68.

\bibitem{DLMP} D. T. Hoang, V. E. Levit, E. Mandrescu,    M. H. Pham, \emph{On
the unimodality of the independence polynomial of clique corona graphs}
  Available online at SSRN: \url{http://dx.doi.org/10.2139/ssrn.4293649}%
.

\bibitem{HLMP2024} D. T. Hoang, V. E. Levit, E. Mandrescu,    M. H. Pham, \emph{%
\ Log-concavity of the independence polynomials of $\mathbf{W}_{p}$ graphs}, 
  \url{https://doi.org/10.48550/arXiv.2409.00827}.

\bibitem{HT16} D. T. Hoang,    T. N. Trung, \emph{A characterization of
triangle-free Gorenstein graphs and Cohen--Macaulayness of second powers of
edge ideals}, Journal of Algebraic Combinatorics, \textbf{43} (2016) 325--338.


\bibitem{HT18} D. T. Hoang,    T. N. Trung, \emph{Buchsbaumness of the second powers of edge ideals}, Journal of Algebra and Its Applications, \textbf{17} (6) (2018), 1850117.  

\bibitem{JV21}  D.  Jaramillo,    R. H. Villarreal, {\it The v-number of edge ideals}, Journal of Combinatorial Theory. Series A,
 {\bf 177} (2021) 105310.  

\bibitem{LM2017} V. E. Levit,    E. Mandrescu, \emph{The Roller-Coaster
conjecture revisited}, Graphs and Combinatorics, \textbf{33} (2017)
1499--1508.

\bibitem{LM19} V. E. Levit,    E. Mandrescu, $1$\emph{-well-covered graphs
revisited}, European Journal of Combinatorics, \textbf{80} (2019) 261--272.

\bibitem{Pinter} M. R. Pinter,\textit{\ }\emph{A class of planar
well-covered graphs with girth four}, Journal of Graph Theory, \textbf{19}
(1995) 69--81.

\bibitem{Pinter2} M. R. Pinter, \emph{Planar regular one-well-covered graphs}%
, Congressus Numerantium, \textbf{91} (1992) 159--159.

\bibitem{Pinter1991} M. R. Pinter, $\mathbf{W}_{2}$\emph{\ graphs and
strongly well-covered graphs: two well-covered graph subclasses}, Vanderbilt
Univ. Dept. of Math. Ph.D. Thesis, 1991.

\bibitem{Plummer1970} M. D. Plummer, \emph{Some covering concepts in graphs}%
, Journal of Combinatorial Theory, \textbf{8} (1970) 91--98.

\bibitem{Plum93} M. D. Plummer, \emph{Well-covered graphs: survey},
Quaestiones Mathematicae, \textbf{16} (1993) 253--287.

\bibitem{Plum67} M. D. Plummer, \emph{On a family of line-critical graphs}, Monatshefte f\"ur Mathematik, {\bf 71} (1) (1967), 40--48. 

  

\bibitem{StaplesThesis} J. W. Staples, \emph{On some subclasses of
well-covered graphs}, Ph.D. Thesis, 1975, Vanderbilt University.

\bibitem{Staples} J. W. Staples, \emph{On some subclasses of well-covered
graphs}, Journal of Graph Theory, \textbf{3} (1979) 197--204.

\bibitem{TV92} J.~Topp,    L.~Volkman, \emph{On the well--coveredness of
products of graphs}, Ars Combinatoria, \textbf{33} (1992) 199--215.





\bibitem{W09} R. Woodroofe, \emph{Vertex decomposable graphs and obstructions to shellability}, Proceedings of the American Mathematical Society,
  {\bf 137} (2009) 3235--3246. 



\end{thebibliography}
\end{document}